\documentclass[12pt]{amsart}

\usepackage{amsfonts,amsmath,latexsym,amssymb,verbatim,amsbsy,amsthm}

\usepackage[top=1in, bottom=1in, left=1in, right=1in]{geometry}

\usepackage[dvipsnames]{xcolor}

\theoremstyle{plain}
\newtheorem{THEOREM}{Theorem}[section]

\newtheorem{corollary}[THEOREM]{Corollary}
\newtheorem{lemma}[THEOREM]{Lemma}
\newtheorem{proposition}[THEOREM]{Proposition}

\theoremstyle{definition}

\theoremstyle{remark}

\newtheorem{remark}[THEOREM]{Remark}
\newtheorem{claim}[THEOREM]{Claim}

\newcommand{\lem}[1]{Lemma~\ref{#1}}

\newcommand{\cor}[1]{Corollary~\ref{#1}}
\newcommand{\prop}[1]{Proposition~\ref{#1}}

\def \a {\alpha}

\def \d {\delta}
\def \D {\Delta}
\def \g {\gamma}

\def \e {\varepsilon}
\def \f {\varphi}

\def \l {\lambda}

\def \n {\nabla}
\def \s {\sigma}
\def \th {\theta}

\def \o {\omega}
\def \O {\Omega}

\def \bs {\boldsymbol{\sigma}}

\def \cH {\mathcal{H}}

\newcommand{\N}{\ensuremath{\mathbb{N}}}   
\newcommand{\Z}{\ensuremath{\mathbb{Z}}}   
\newcommand{\R}{\ensuremath{\mathbb{R}}}   
\newcommand{\T}{\ensuremath{\mathbb{T}}}   

\renewcommand{\S}{\ensuremath{\mathbb{S}}} 

\def \lan {\langle}
\def \ran {\rangle}
\def \p {\partial}
\def \ra {\rightarrow}

\def \bs {\backslash}

\DeclareMathOperator{\sign}{sign} %
\DeclareMathOperator{\diver}{div} %
\DeclareMathOperator{\curl}{curl} %
\DeclareMathOperator{\rres}{res} %

\def \eth {\vec{e}_\th}
\def \nor {\vec{n}}
\def \ef {\vec{e}_\phi}
\def \er {\vec{e}_r}
\def \const {\mathrm{const}}
\def \loc {\mathrm{loc}}
\def \Vol {\mathrm{Vol}}

\begin{document}

\title{Homogeneous solutions to the 3D Euler system}
\author{Roman Shvydkoy}
\address{Department of Mathematics, Statistics, and Computer Science
University of Illinois at Chicago
322 Science and Engineering Offices (M/C 249)
851 S. Morgan Street
Chicago, IL 60607-7045}
\email{shvydkoy@uic.edu}

\thanks{The work of the author is partially supported by NSF grants DMS-1210896 and DMS-1515705.}

\begin{abstract} We study stationary homogeneous solutions to the 3D Euler equation. The problem is motivated be recent exclusions of self-similar blowup for Euler and its relation to Onsager conjecture and intermittency. We reveal several new classes of solutions and prove rigidity properties  in specific categories of genuinely 3D solutions. In particular, irrotational solutions are characterized by vanishing of the Bernoulli function; and tangential flows are necessarily 2D axisymmetric pure rotations. In several cases solutions are excluded altogether. The arguments reveal geodesic features of the Euler flow on the sphere. We further show that in the case when homogeneity corresponds to the Onsager-critical state, the anomalous energy flux at the singularity vanishes, which is suggestive of absence of extreme $0$-dimensional intermittencies in dissipative flows.
\end{abstract}



\maketitle
\section{Introduction} We study classification problem of stationary homogeneous solutions to the Euler system given by 
\begin{equation}\label{e:ES}
\begin{split}
V \cdot \n V + \n P &= 0 \\
\diver V &  = 0.
\end{split}
\end{equation}
This is a classical system that describes time independent motion of an incompressible ideal fluid in $\R^3$, where $V$ is the velocity field and $P$ is the pressure. Scaling symmetries of the system, namely $V \to a V(bx)$, $P \to a^2 P(bx)$, allow for possible existence of invariants which are homogenous solutions of the form
\begin{equation}\label{e:VP}
\begin{split}
V(x)  = \frac{v + f \nor}{|x|^\a}, \quad P(x) = \frac{p}{|x|^{2\a}}.
\end{split}
\end{equation}
Here $v$ is the tangent component of $V$ on the sphere $\S^2$, $f$ is normal ($\nor$ denotes the outward unit normal), and $p$ is the spherical pressure. We are only concerned with $C^1$-solutions (at least) for which $v,f,p \in C^1(\S^2)$ and the system \eqref{e:ES} can be understood classically in $\R^3\bs\{0\}$. 

Out motivation for studying homogeneous solutions, apart from purely academic standpoint, comes from three different sources. First, recent studies of self-similar blowup for the full dynamical Euler equation demonstrated that under a mild growth restriction on the profile $V$, $V$ necessarily behaves like $\frac{1}{|x|^\a}$ at infinity, see \cite{BrShv,ChShv}. This suggests that homogeneous solutions are the only ones that exist in the class of self-similar. Second, in the case $\a = \frac23$ (or $\a = \frac13$ in 2D), the field \eqref{e:VP} gains so-called Onsager critical regularity $B^{1/3}_{3,\infty}$ near the origin.  Such regularity allows for the energy balance law to break and it is relevant in newly emerged descriptions of turbulent flow (see \cite{onsager,ds,shv-lectures,isett,CS}). The distinctive feature of being singular only at one point makes homogeneous solution a viable candidate for a dissipative flow with extreme $0$-dimensional intermittency, \cite{CS}. Thirdly, in the context of viscous flows, where $\a=1$ is the only relevant scaling,  Landau revealed in 1944 a class of homogeneous solutions with axial symmetry,  \cite{landau,batchelor}. They appear naturally in studying jet flows emanating from a point source. Recently, V.~{\v{S}}ver{\'a}k demonstrated in \cite{sverak} that Landau solutions are the only ones in the class of all homogeneous solutions. This motivates us to look for similar rigidity properties in the inviscid case, which turn out to be abundant. The question of vanishing viscosity limit  also comes into focus and we address it in Section~\ref{s:landau}.

In recent work \cite{2d-homo} we studied homogeneous solutions in $\R^2$ and provided full classification in all cases except $\a \in ( - \frac13,\frac14)\bs \{0\}$. If embedded in $\R^3$ the solutions are $C^1$-smooth on the sphere only for $\a \leq -1$. In this range  we can provide a complete description which we summarize in Section~\ref{s:2D}. In this present paper we focus on genuinely 3D case. It appears that 3D solutions come in classes with manifestly rigid character, in contrast to 2D case. The main reason consist in the fact that $\S^2$ has trivial first DeRham group, while in $\S^1$ existence of harmonic fields results in the class of 2D circular rotational solutions and overall allows more topological freedom for other elliptic solutions to exist. As a consequence, we show that there are no $C^1$-smooth solutions for $\a=1$, \prop{p:a=1}\footnote{During the preparation of the paper the author was informed that this particular result also appeared independently in \cite{LX}. See Section~\ref{s:eq} for discussion.}. Furthermore, we exclude axisymmetric  solutions in the wider range $0<\a<2$ in \prop{p:noaxi}. In the way of our analysis we draw connection with the Landau solutions. We show that they necessarily have to loose regularity  for positive  values of $\nu>0$ in the process as they converge to Euler solutions. Several  new examples of genuinely 3D solutions are exhibited. Those are $2\frac12$-dimensional solutions obtained from 2D ones by attaching a passive third component, Section~\ref{s:212}; geodesic solutions with straight particle trajectories, in particular, parallel shear \eqref{e:psf1}, radial \eqref{e:rad}, and axisymmetric conical solutions with or without swirl \eqref{e:Vaxi}. The latter is a new class of explicit examples of stationary axisymmetric flow. In addition, we discover an important class of irrotational solutions obtained by setting $f=Y_l^m$, one of spherical harmonics, and $v = (1-\a)^{-1} \n f$.  Here $\a \in \Z \bs \{1\}$. This class is a direct analogue of the classical point vortices in 2D.  These are also the only smooth solutions we discovered that include positive values of $\a$. We then establish a number of rigidity results that give a simple characterization of the above constructed solutions. Namely, we show that the Bernoulli function $H = |v|^2 + f^2 + 2p$, which play a crucial role in all our analysis, vanishes for all irrotational flows, and for $\a \leq 2$ any solution with $H=0$ is necessarily irrotational. Recall that for a general steady state vanishing $H$ characterizes all Beltrami solutions. So, we observe exclusively the effect of homogeneity. Next, axisymmetric solutions with constant spherical pressure $p$ are necessarily geodesic and are all described by the class \eqref{e:Vaxi} and \eqref{e:rad}, see \prop{p:axip}. We found two first integrals for the $4\times 4$-system of ODE describing such solutions, which leads to a complete resolution in this particular case. Lastly, we establish rigidity of all tangential solutions: if $f = 0$ throughout, then there is an axis of rotation around which the solution is the 2D purely rotational state given by \eqref{e:rotCar}. This once again stresses the difference between 2D and 3D cases and reveals inherently geodesic nature of the Euler flow on the sphere.

In the Onsager-critical case of $\a = \frac23$ we prove that the solution, properly tapered at infinity, regains finite global energy $\|V\|_2 <\infty$, which introduces a physically reasonable force $F$ in the system \eqref{e:ES}, namely $F\in C^\infty_{\loc}$ and $|\n^k F(x) |\lesssim \frac{1}{|x|^{3+k}}$. The classical Onsager conjecture inquires whether such solutions may have anomalous energy flux, which in steady case amounts to nonvanishing work of force, $\Pi = \int V \cdot F dx$. We show via an approximation argument that in fact $\Pi = 0$. As argued in \cite{CS,shv-lectures} such solutions present an extreme case of intermittent state where energy flux concentrates on a $0$-dimensional set, namely the origin. In 2D, we observed in \cite{2d-homo} that the Hamiltonian structure of the reduced equations on the sphere $\S^1$ produces extra symmetry in solutions that ultimately causes vanishing of the flux. In 3D case such constructive explanation remains to be found, despite the fact that we can formally prove vanishing of the flux in this case also. Our overall message with regard to the Onsager case is that the Euler system may not support extremely intermittent dissipative solutions due to hidden symmetries. In comparison, all ``wild" solutions constructed with the use of the new convex integration technique (see \cite{ds,isett}) have no intermittency, with dimension of singularity set being $3$, the entire domain.

Based on the constructed examples in this paper, their rigidity, and the flavor of some of our arguments we conjecture that there are no $C^1$-smooth solutions in the range $\a >-1$, except the irrotational ones at $\a \in \Z\bs\{1\}$. We also claim that the maximal smoothness of solutions behaves like $C^{-\a}$ for $\a <0$.

\section{Equations on the sphere and examples}\label{s:eq}
The Euler system of equations \eqref{e:ES} for homogenous solutions reduces to the following system on $\S^2$ (see Appendix):
\begin{subequations} \label{e}
\begin{align}
(2-\a)f+ \diver v & = 0 \label{e1} \\
v \n f &= |v|^2+ \a f^2 + 2\a p \label{e2}\\
(1- \a)f v + v \n v & = -\n p \label{e3}.
\end{align}
\end{subequations}
We study solutions for which the system \eqref{e} can be understood classically, i.e. when all ingredients $v,f,p$ belong to $C^k(\S^2)$ for some $k \geq 1$. We call these cumulatively $C^k$-solutions. 
The system \eqref{e} can be written in a fixed spherical system of coordinates
\begin{equation}\label{}
\begin{split}
x&= \sin \phi \cos \th, \quad y = \sin \phi \sin \th, \quad z = \cos \phi, \\
v &=  a \ef + b \eth
\end{split}
\end{equation}
where $\ef, \eth$ are the vectors of standard orthonormal frame associated with $(\phi, \th)$,  as follows
\begin{equation}\label{e:sph}
\begin{split}
(2-\a) f + a_\phi + a \cot \phi\,+ b_\th(\sin\phi)^{-1}  & = 0\\
 a f_\phi +  b f_\th (\sin\phi)^{-1} &=  a^2+b^2 +\a f^2+ 2\a p  \\
(1-\a)f a \sin \phi + a a_\phi \sin \phi  +  b a_\th - b^2 \cos \phi & = -  p_\phi \sin \phi \\
(1-\a)f b \sin\phi  +  a b_\phi \sin \phi + b b_\th +  ab \cos \phi&= - p_\th
\end{split}
\end{equation}
where subindices stand for partial derivatives. This somewhat obscure form of the system will be useful in obtaining and classifying various special classes of solutions. 

Let us introduce an important geometric quantity that will play a crucial role in what follows. The restriction of the classical Bernoulli function $2P + |V|^2$ on the sphere take form $H = |v|^2 + f^2 + 2p: \S^2 \ra \R$. Multiplying \eqref{e3} with $v$ and using \eqref{e2} we obtain the following transport equation for $H$:
\begin{equation}\label{e:H}
v \n H = 2 \a fH.
\end{equation}
As a consequence of \eqref{e}, \eqref{e:H} we will obtain an exclusion of smooth solutions in the case $\a=1$. Note that this appears to be in complete contrast to the Navier-Stokes system, where $\a=1$ is the \emph{only} possible scaling for homogenous solutions to exist. Before we prove the result let us rewrite the momentum equation \eqref{e3} in terms of $H$. First, let us consider the vorticity of $v$, $\o = \curl v$, or formally, $\o = \star d v^\flat$ (we use standard notation for operations on a manifold, see Appendix). One can easily verify using normal coordinates that $v\n v - \frac12 \n |v|^2 = \o v^\perp$. Here, $\perp$ means rotation of $v$ counterclockwise by $90^\circ$ relative to the outward oriented normal, or formally, $v^\perp = (\star v^\flat)^\sharp$. We will drop superindecies $\flat, \sharp$ in the future for brevity. So, \eqref{e3}  becomes
\begin{equation}\label{e:H2}
(1-\a)f v + \o v^\perp = -\n(p + \frac12 |v|^2)
\end{equation}
and in terms of $H$,
\begin{equation}\label{e:H3}
f( (1-\a) v - \n f) + \o v^\perp = -\frac12 \n H.
\end{equation}

\begin{proposition}\label{p:a=1} There are no $C^1$-solutions to the system \eqref{e} for $\a = 1$.
\end{proposition}
\begin{proof}
In the case $\a = 1$, \eqref{e1}, \eqref{e2}, and \eqref{e:H} read
\begin{equation}\label{e:a=1}
f+ \diver v  = 0, \quad v \n f = H, \quad v \n H = 2 fH.
\end{equation}
Let us test the last one with  $f$  and integrate by parts:
\[
\int f v \n H d\s= \int f^2 H d\s- \int H v \n f d\s= 2 \int f^2 Hd\s.
\]
So, using the second in \eqref{e:a=1},
\[
\int f^2 H d\s= - \int H v \n f d\s= - \int H^2 d\s.
\]
Again, from the second equation in \eqref{e:a=1},
\[
\int f^2 H d\s= \int f^2 v \n f d\s= \frac13 \int v \n f^3 d\s= \frac13 \int f^4  d\s.
\]
We have obtained $- \int H^2 d\s= \frac{1}{3} \int f^4 d\s$.
So, $H = f = 0$. From \eqref{e:H3}, we obtain $gv = 0$. This implies that $d v = 0$ on the set where $v \neq 0$, and hence $d v = 0$ on the closure $\overline{\{ v \neq 0\}}$. But on the complement of the closure, $v = 0$ trivially. Consequently, $d v = 0$ throughout, and also $\d v = 0$. We conclude that $v$ is harmonic as a form, and hence $0$.
\end{proof}
As we noted in the introduction this result also appeared independently in \cite{LX}. The argument of \cite{LX} is based on analysis from the bulk of the fluid domain $\R^3$, as opposed to the sphere. However, in both cases the conclusion is finding that $v$ is harmonic.

\subsection{Geodesic solutions} Many explicit examples of homogeneous solutions have flow trajectories that are straight lines (or rays) in space. We call these geodesic solutions. The geodesic property can be expressed concisely by $V\cdot \n V \times V = 0$.  If the pressure $P$ is constant, then clearly $V$ is geodesic. Constant $P$ implies $p = 0$ for $\a \neq 0$, and $p = \const$ for $\a = 0$. In any case, $p$ disappears from the system \eqref{e}. Then \eqref{e3} implies that the orbits of $v$ on the sphere are geodesic too. One simple example is given by the parallel shear flow
\begin{equation}\label{e:psf1}
V = \lan 0,0, \frac{z(\th)}{r^\a} \ran,
\end{equation}
where $r = \sqrt{x^2+y^2}$ and $z\in C^1(\T)$. This is a $C^1$-smooth solution for $\a \leq -1$. It will be crucial to  catalogue solutions in terms of their spherical quantities, even if it may not always be most illuminating. It will help compare them with other solutions obtained solely in terms of $f,v$, etc. Thus, we have
\begin{equation}\label{e:psf2}
\begin{split}
f & = z(\th) \frac{\cos \phi}{\sin^\a \phi}, \quad b = 0,\quad a = - z(\th) \sin^{1-\a}\phi \\
p & = 0, \quad H = \frac{z^2(\th)}{\sin^{2\a}\phi}. 
\end{split}
\end{equation}

Another simple example is the radial flow
\begin{equation}\label{e:rad}
\a = 2, \quad f= \const, \quad v=0, \quad p= - \frac12 f^2.
\end{equation}
This solution is unique in several different categories.  It is the only one for which the tangential ingredient of $H$, $2p + |v|^2$ is constant (see \cor{c:a=2}); is the only axisymmetric solution in the scaling $\a = 2$ (see Section~\ref{s:axi}); and is the only solution in the general radial class. For the latter, if $v = 0$, then from \eqref{e1} we have $\a = 2$ for otherwise $f=0$ and the solution is trivial. Momentum equation \eqref{e3} implies $p = \const$, and hence from \eqref{e2} so is $f$. Note that this is an example of a geodesic solution for which the global pressure $P$ is not constant. 

A class of axisymmetric solutions with or without swirl can be constructed as follows. Let $a_0,b_0$ with $a_0^2 + b_0^2 = 1$ represent local spherical coordinates of the tangent field $v$ on the equator, $v_0 = a_0 \ef + b_0 \eth$ (see Appendix). Then $V = \lan V^x, V^y, V^z \ran$ is given by
\begin{equation}\label{e:Vaxi}
\begin{split}
V^x & = b_0^2 \frac{xz}{x^2 + y^2} K^{-\frac{\a}{2}} + b_0 \frac{y}{x^2+y^2} K^{1-\frac{\a}{2}}, \\
V^y & = b_0^2 \frac{yz}{x^2 + y^2} K^{-\frac{\a}{2}} - b_0 \frac{x}{x^2+y^2} K^{1-\frac{\a}{2}}, \\
V^z & = a_0^2 K^{-\frac{\a}{2}},\\
\a p & = 0.
\end{split}
\end{equation}
where 
\begin{equation}\label{}
K = \left\{ \begin{split}
a_0^2 (x^2+y^2) - b_0^2 z^2, & \quad b_0^2z^2 \leq a_0^2(x^2 + y^2),\\
0, & \quad  b_0^2z^2 > a_0^2(x^2 + y^2).
\end{split} \right.
\end{equation}
So, in this case the swirl $b_0$ determines the aperture of the cone where $V$ vanishes. Clearly, the solution with a swirl is $C^1$ only in the range $\a \leq -2$, and in the range $\a \leq -1$ without swirl. In both case, this also implies $p=0$. As shown in Proposition \ref{p:axip} these are unique solutions in the class of axisymmetric solutions with constant spherical pressure $p$.

\subsection{2D homogeneous solutions}\label{s:2D} A large class of solutions can be obtained by lifting the 2D homogeneous solutions into space.  The 2D case has been classified in \cite{2d-homo}. Let us give a brief recitation of the obtained results as it would provide some valuable insight into existing possibilities. In a fixed coordinate system $(x,y,z)$ the 2D homogenous solutions are  given by
\[
V(r,\th) = \frac{u_\th(\th) \eth + u_r(\th) \er}{r^\a}, \quad P(r,\th) = \frac{p(\th)}{r^{2\a}},
\]
where $\eth$ and $\er$ are unit basis vectors associated with the polar system. Such solutions gain $C^1$-regularity only for $\a \leq -1$ because of singular behavior at the poles. One can associate a stream-function to the field $V = \n^\perp \Psi$ given by $\Psi = r^{1-\a} \psi(\th)$,  $u_\th = (1-\a)\psi$, $u_r = -\psi'$. In our spherical system, we have
\begin{equation}\label{}
\begin{split}
f & = -\psi'(\th) \sin^{1-\a} \phi, \ a = - \psi'(\th) \frac{\cos\phi}{\sin^\a\phi},\ b= (1-\a) \frac{\psi(\th)}{\sin^{\a} \phi} \\
H & = \frac{2p+(1-\a)^2 \psi^2+(\psi')^2}{\sin^{2\a}\phi}.
\end{split}
\end{equation}

A complete classification of solutions in the range $\a \leq -1$ is given in \cite{2d-homo}. We will summarize the results as they would provide some valuable insight into existing possibilities. The Euler system forces $p$ to be constant, and satisfy 
\begin{equation}\label{e:2D}
\begin{split}
 - 2 \a p & =\a (\psi')^2 +(1-\a)^2 \psi^2 + (1-\a) \psi'' \psi,\\
\psi(0) & = \psi(2\pi).
\end{split}
\end{equation}
The ODE has a conserved quantity (coming from conservation of the Bernoulli function along particle lines), 
\begin{equation}\label{B}
B = (2p+(1-\a)^2 \psi^2+(\psi')^2)\psi^{\frac{2\a}{1-\a}}.\end{equation}
With this law system \eqref{e:2D} becomes a Hamiltonian system in phase variables $(x,y) = (\psi,\psi')$ given by
\begin{equation}\label{sys-B}
\left\{\begin{split}
		x' & = y\\
		y' & = -(1-\a)^2x +\frac{\a}{\a-1} B x^{\frac{\a+1}{\a-1}}.
\end{split}\right.
\end{equation}
with the pressure $p = - \frac{y^2}{2} - \frac{(1-\a)^2}{2} x^2 + \frac{B}{2} x^{\frac{2\a}{\a-1}}$ being the Hamiltonian. Thus, the question reduces to finding $2\pi$-periodic solutions. Explicit formulas for those solutions are not always available however we can classify and count all types of solutions that exist. Solutions with $\psi > 0$ have elliptic-type streamlines, therefore called elliptic, solutions with vanishing $\psi$ at two or more points have hyperbolic streamlines. Parabolic solutions don't exist in our range $\a \leq -1$. Elliptic ones correspond to $p>0$, $B>0$, while hyperbolic to $p<0$ and arbitrary $B$. Hyperbolic solutions always hit zero at the same slope up to a sign, namely, $\psi' = \pm \sqrt{-2p}$. Pieces of $\psi$ over sign-definite intervals can be separated, flipped, and glued together to form new solutions as long as they correspond to the same pressure $p$. Thus, hyperbolic pieces of  $\psi$ must alternate signs in order to produce $C^1$-solutions. So, classification in hyperbolic case reduces to finding time-span function $T=T(p,B)$ that measures the length of intervals of sign-definiteness of $\psi$. Rescaling $\psi$ by a constant allows to reduce the question to a fixed $p=-1,0,1$ or $B=-1,0,1$. 

In the elliptic case we have the following description.  Since, $p\geq 0$, then $B >0$. Rescale $B$ to $B=1$. Then for $p=0$ all solutions are parallel shear flows. For $p=p_{\mathrm{max}} = \frac{1}{2(1-\a)} \left( \frac{\a}{(\a-1)^3} \right)^{-\a}$ the solution is pure rotation, $\psi = \const$. For $0<p<p_{\mathrm{max}}$ in the range $-\frac72 \leq \a <-1$ there are no elliptic solutions. In the range $\a < - \frac72$ there are exactly $\# \{ (2,\sqrt{2(1-\a)}) \cap \N \}$ of non-trivial elliptic solutions. For $\a = -1$, the exceptional case, all solutions for $0<p<p_{\mathrm{max}}$ are $2\pi$-periodic and given explicitly by 
$\psi = \g_1 + \g_2 \cos(2\th)$, $p = 2(\g_1^2 - \g_2^2)$, and $\Psi = (\g_1+\g_2)x^2+ (\g_1-\g_2)y^2$. Thus all streamlines are perfect ellipses in this case. 

In the hyperbolic case, we rescale $p=-1$, so that all pieces can be stitched to form a $C^1$-solution. Then for $B>0$ we have $\frac{\pi}{1-\a} < T < \pi$ and $T$ changes monotonely; $B=\infty$ corresponds to already accounted parallel shear flow with $T = \pi$; $B = 0$, $T = \frac{\pi}{1-\a}$; and for $B<0$, $0<T<\frac{\pi}{\l}$. Clearly, there are infinitely many possibilities for $T$'s to add up to a full $2\pi$-period. Conversely, all hyperbolic solutions are obtained this way. 

The case $B=0$ is exceptional because in this case the vorticity $\o =  r^{-1-\a}( (1-\a)^2 \psi + \psi'')$ vanishes. The flow in the corresponding sector is irrotational. We will see that irrotational solutions are indeed unique in the class of solutions with vanishing  Bernoulli function in the range $\a \leq 2$.

\subsection{$2\frac12$D homogeneous solutions}\label{s:212} The classical way to construct a 3D solution out of a 2D solution $U = \lan u_1, u_2, 0 \ran$ is to attach a third component $Z$ which is transported along $U$. To satisfy homogeneity we set $Z = \frac{z(\th)}{r^\a}$. The transport requires $U \cdot \n Z = 0$. In terms of the stream-function this condition takes form  
\[
\a \psi' z + (1-\a) z' \psi = 0,
\]
and hence
\begin{equation}\label{}
|\psi|^\a|z|^{1-\a} = \const.
\end{equation}
The constructed solutions have the same constant spherical pressure as the underlying 2D solution. The other spherical quantities are superpositions of the previous two examples. In particular,
\begin{equation}\label{}
H  = \frac{2p+(1-\a)^2 \psi^2+(\psi')^2 + A |z|^{\frac{2\a}{\a-1}}}{\sin^{2\a}\phi},
\end{equation}
where $A>0$ is a constant.

\section{Irrotational solutions}

Let us first discuss the structure of vorticity. Let $ \O = \n \times V$ be the classical vorticity in $\R^3 \bs \{0\}$. Denote $u=(1-\a) v^\perp - \n^\perp f$. We have the following expression for $\O$: 
\begin{equation}\label{}
\O = \frac{1}{|x|^{\a+1}}(u + \o\, \nor).
\end{equation}
Since $\O$ is divergence-free, we obtain the relationship
\begin{equation}\label{}
 (1-\a) \o + \diver u = 0.
\end{equation}
In terms of $\O$, the Euler system takes form
\begin{equation}\label{e:OV}
\O \times V = - \frac{1}{2} \n (|x|^{-2\a} H).
\end{equation}
Reading off the normal and tangental part of this identity we obtain the following system
\begin{subequations}\label{e:uv}
\begin{align}
u \times v& = \a H \nor \label{e:uv1}\\
f u - \o v & = -\frac12 \n^\perp H.\label{e:uv2}
\end{align}
\end{subequations}
Here, equation \eqref{e:uv2} is clearly equivalent to \eqref{e:H3}, while equation \eqref{e:uv1} is in fact \eqref{e2} in disguise. It can be obtained from \eqref{e2} by using the identities $ v^\perp \times v = - |v|^2 \nor$, and $\n^\perp f \times v = -v\n f \nor$. At least when $\a \neq 0$ equation \eqref{e:uv1} reveals the obvious geometric interpretation of the Bernoulli function. It also implies that $H$ should vanish at some point, unless $\a =0$.

\begin{proposition} Suppose $v,f,p \in C^1(\S^2)$, and $\O = 0$. Then $\a \in \Z \bs \{1\}$ and the solution is given by   
\begin{equation}\label{e:irr}
(1-\a) v = \n f,\quad p = - \frac12 f^2 - \frac{1}{2(1-\a)^2} |\n f|^2,
\end{equation}
where $f$ is a constant multiple of one of the spherical harmonics $Y_l^m$, $1-\a = l$, $-l \leq m \leq l$. Moreover, in this case $H = 0$.
\end{proposition}
\begin{proof} Vanishing of $\O$ immediately implies $u = 0$, which implies \eqref{e:irr}. By taking the divergence of \eqref{e:irr} and combining with \eqref{e1} we obtain the classical eigenvalue problem for the Laplace-Beltrami operator
\begin{equation}\label{e:Ha}
\Delta f = -(2-\a)(1-\a) f = -(l+1)l f.
\end{equation}
The description of $f$  follows automatically. The pressure is recovered directly from \eqref{e2}.
\end{proof}
Note that when $\a = 2$ the only irrotational flow is the radial one \eqref{e:rad}.

Let us  take the curl of \eqref{e:OV} in $\R^3 \bs\{0\}$. We obtain the classical  vorticity equation
\begin{equation}\label{e:OV2}
[\O, V] = 0.
\end{equation}
On the sphere it takes the form
\begin{subequations}\label{e:lie}
\begin{align}
u\n v - v\n u &= (1+\a) \o v - (2+\a) f u,\label{e:lie1} \\
v\n \o - u\n f &= f\o. \label{e:lie2}
\end{align}
\end{subequations}
Here \eqref{e:lie1} represents the tangential, and \eqref{e:lie2} represents the normal components of \eqref{e:OV2}. The latter is not independent -- it can also be obtained by taking the divergence of \eqref{e:lie1}.

\begin{proposition}\label{p:irrot}
For $\a \leq 2$ irrotational solutions are unique in the class of all $C^2$-smooth solutions with $H = 0$. For $\a=0$ irrotational solutions are unique in the class of all $C^2$-smooth solutions with $H=\const$. 
\end{proposition} 
\begin{proof} The case $1 <\a \leq 2$ is actually straightforward. We have from \eqref{e2}, 
\begin{equation}\label{e:evaux}
v \n f = (1-\a) |v|^2.
\end{equation}
Let us integrate over $\S^2$ and integrate by parts on the left. Using \eqref{e1} we obtain
\begin{equation}\label{}
(2-\a) \int f^2 d \s = (1-\a) \int |v|^2 d\s. 
\end{equation}
This implies $f , v=0$ unless $\a =2$ in which case we obtain the radial irrotational solution $v=0$, $f = \const$.

Let us turn to the case $\a <1$ ($\a = 1$ having been excluded). From \eqref{e:uv2} we obtain $fu = \o v$ for any constant $H$. Also, \eqref{e:evaux}  holds for zero $H$ or constant $H$ with $\a=0$. Using \eqref{e:lie2}  in addition, we have for all $n\in \N$ the identity
\[
\diver(f \o^n v) = (n+\a-2) f^2 \o^n + (1-\a)(n+1) |v|^2 \o^n.
\]
When $\a<1$ we can choose a large even $n$ for which the right hand side is pointwise non-negative. Integrating over the sphere we see that it must vanish pointwise. This implies that if $\o \neq 0$, then $f=v=0$ at the same point. In either case, $\o v = fu =0$ throughout. Thus, on the set $\{f \neq 0\}$, we have $u =0$, i.e. $\n f = (1-\a) v$. Taking the divergence we obtain the Laplace equation \eqref{e:Ha}. By continuity, \eqref{e:Ha} holds on the closure of the set $\{f \neq 0\}$. But on the complement of the closure, \eqref{e:Ha} holds trivially as both sides vanish. So, unless $f$ vanishes identically, in which case we have $v=0$ from \eqref{e:evaux}, $f$ satisfies \eqref{e:Ha} throughout. Hence $f = Y_l^m$, and we know that harmonics do not vanish on a dense set. This in turn implies $u = 0$ everywhere, and hence, $\O = 0$.

\end{proof}

\section{Rotational solutions}
The opposite extreme to radial, and in a sense to irrotational flows altogether, are tangential flows, i.e. ones with the orbits of $V$ living on concentric spheres around the origin. This is only possible when $f=0$ throughout, and hence $\diver v = 0$ (so, as a form $v$ is co-exact as opposed to irrotational exact forms). One obvious example is given from the class of 2D flows as discussed above. Namely, in a fixed Cartesian system, we have
\begin{equation}\label{e:rotCar}
V = \frac{A}{r^{\a+1}} \left \lan-y,x, 0 \right\ran, \quad P = -\frac{A}{2\a r^{2\a}}, \quad r = \sqrt{x^2 + y^2}.
\end{equation}
Note that it gains $C^1$-smoothness only for values $\a \leq -1$. We now show that these are the only examples of $C^1$ tangential solutions.

\begin{proposition}\label{p:f=0} Suppose $f =0$ and $v,p \in C^1(\S^2)$. Then up to a rotation the solution is given by \eqref{e:rotCar}, and  $\a \leq -1$. There are no $C^1$-solutions with $f=0$ for $\a > -1$.
\end{proposition}
\begin{proof} 
In the case $\a=0$ the statement is trivial from \eqref{e2}. We assume that $\a \neq 0$. According to \eqref{e:H}, $H$ remains constant along the orbits of $v$. Furthermore, from \eqref{e2}, 
\begin{equation}\label{e:Hvp}
H = \frac{\a - 1}{\a} |v|^2  = 2(1-\a) p.
\end{equation}
Hence, $|v|^2$ and $p$ are transported as well. Let $x_0$ be a point where $|v|$ attains its maximum, and let $x(t)$ be the $v$-orbit through $x_0$. Since $|v|$ is transported, it will preserve its extreme status, and hence $\n |v|^2 = \n p =0$ on the orbit. Returning to \eqref{e3} we see that $x(t)$ is a complete geodesic. 
Let us denote it $E$. 

From the momentum equation \eqref{e3} and \eqref{e:Hvp} we obtain
\[
v\n v = - \frac{1}{2(1-\a)} \n H.
\]
Taking the $\perp$ and using \eqref{e:uv2} we obtain
\[
v \n (v^\perp) = - \frac{1}{(1-\a)} \o v.
\]
Consequently, $v\n u = - \o v$. Plugging this into \eqref{e:lie1} we obtain $u \n v = \a \o v$.  And finally, taking $\perp$ again,
$u \n u = \a \o u$. Reparametrizing the field $u$ along its own trajectories by $\exp\{ -\a \int_0^t \o(s) ds \} u$ we see that the trajectories are geodesics provided initial $u$ is not zero. On the equator $E$ all vectors of $u$ will point either due north or due south.  This in turn implies that at least in a neighborhood $\Sigma$ of the equator $E$ where $u \neq 0$ the field $u$ points along the meridians. Let us fix spherical coordinates so that
$E = \{ \phi = \pi/2\}$. Then the field $v$ has zero $\p_\phi$-component, and the orbits of $v$ are latitudes. Moreover, since $|v|$ is preserved along $v$-orbits, $v$ is independent of $\th$. According to our conclusions, we have  $f = a = 0$ and $b, p \in C^1$  depend only on $\phi$. In this case, the system \eqref{e:sph} reduces to
\begin{equation}\label{e:b=0}
\begin{split}
 b^2 + 2\a p & = 0\\
b^2 \cot \phi& 
 =  p_\phi
\end{split}
\end{equation}
For $\a = 0$ there are only trivial zero solutions. Otherwise, the solutions are given by
\begin{equation}\label{e:rot}
b = \frac{A}{\sin^{\a} \phi}, \quad p = - \frac{A}{2\a \sin^{2\a} \phi}, \quad A \in \R.
\end{equation}
In Cartesian coordinates this is nothing other than \eqref{e:rotCar}. It also shows that $\Sigma$ covers the entire sphere except poles, and the proposition is proved.
\end{proof}

Another characteristic feature of rotational flows is that $|v|^2 + 2\a p = 0$. It can be shown to be their exclusive property.

\begin{corollary} If $|v|^2 + 2\a p = 0$, then the flow is rotational.
\end{corollary}
\begin{proof}
Indeed, from \eqref{e2} we have the Riccati equation $f' = \a f^2$. So, unless $\a = 0$, $f = 0$ identically, which implies the conclusion via \prop{p:f=0}. If $\a = 0$, then $v =0 $ by assumption, and hence $f=0$ by the divergence equation \eqref{e1}. 
\end{proof}

Let us point out other corollaries of \prop{p:f=0}.

\begin{corollary}\label{c:a=2} If $|v|^2 + 2p = \const$, then $\a = 2$, $f = \const$, $v=0$.
\end{corollary}
\begin{proof} From \eqref{e:H2} we immediately obtain $(1-\a) f v + g v^\perp = 0$. We can assume that $\a \neq 1$, in which case the above shows that 
\[
v(x) \neq 0 \implies f(x) =g(x) = 0.
\]
Unless $\a = 2$, by continuity and \eqref{e1} this implies that $f = 0$ throughout. By Proposition~\ref{p:f=0} this describes the solution as rotational, which is a contradiction, since for such solutions $2p + |v|^2 \neq \const$ unless $v = p=0$. If $\a = 2$, then by continuity $g = 0$ throughout, and in addition $v$ is divergence-free. So, $v$ is harmonic as a form, hence $v=0$. Then $p$ is a constant, and from \eqref{e2} we conclude that $f = \const$, which identifies the solution as described.
\end{proof}

\begin{corollary} Suppose $p \geq 0$ and $\a >0$. Then the solution in trivial, $v = f = p = 0$.
\end{corollary}
\begin{proof} From  \eqref{e2} we have the Riccati inequality $f' \geq \a f^2$. Unless initial condition is $0$ the solution will blow up either forward of backward along the orbit. This immediately implies $f = 0$ throughout. \prop{p:f=0} finishes the proof.
\end{proof}

In the range $0<\a<1$ we can establish a much stronger statement exploiting the dynamical nature of the system \eqref{e2}, \eqref{e:H}. Let us rewrite it as a system over the trajectories of $v$:
\begin{align}
f_t& = \a H + (1- \a)|v|^2 \label{f}\\ 
H_t& = 2\a f H. \label{H}
\end{align}

\begin{lemma} \label{l:H<0}
In the range $0<\a<1$, we have $H\leq 0$, and hence $p\leq 0$, throughout. 
\end{lemma}
\begin{proof}
Let us fix $x_0 \in \S^2$, and assume that $H_0 = H(x_0) \neq 0$. From \eqref{H} we readily obtain
\[
\begin{split}
H(t) &= H_0 \exp \{ 2\a \int_0^t f(s) ds \} \\
H(-t) &= H_0\exp \{ -2\a \int_0^t f(-s) ds \} 
\end{split}
\]
Suppose $0<\a<1$. Since $H$ is bounded, this implies that 
\[
\begin{split}
\int_0^t f(s) ds  < M;\quad \int_0^t f(-s) ds & > - M
\end{split}
\]
for some $M$ and all $t>0$. So,
\[
\begin{split}
\limsup_{t \ra \infty} f(t) \leq 0; \quad \liminf_{t \ra -\infty} f(t) & \geq 0.
\end{split}
\]
This implies that at some point of time $t^*$, $f_t(t^*) \leq 0$. Hence, from \eqref{f}, 
$H(t^*) \leq 0$. Since the sign of $H$ remains constant along the trajectory, we obtain $H(x_0)\leq 0$.
\end{proof}

We note that sign-definiteness of the Bernoulli function $H$ has been instrumental in establishing Liouville theorems for the axisymmetric solutions to the Navier-Stokes and Euler equations in \cite{kor} , and ruling out higher than $4$-dimensional homogeneous Landau-type solutions for the Navier-Stokes system, \cite{sverak}. In our case the geometric implication of \lem{l:H<0} and \eqref{e:uv1} states that the form $v \wedge u$ is co-oriented with the canonical volume form at any given point on $\S^2$.

\section{Axisymmetric solutions}\label{s:axi}
In this section we study axisymmetric solutions with or without swirl. We assume that $\a \neq 1$ as this case has been ruled out by \prop{p:a=1} as having no smooth solutions. In order for a solution to remain smooth at the pole we necessarily have $a(0) = a( \pi) = b(0) = b(\pi) = 0$.  The system \eqref{e:sph} in our case reduces to
\begin{subequations} \label{e:axi}
\begin{align}
(2-\a) f + a' + a \cot \phi &= 0  \label{axi1}\\
a f' &= a^2 + b^2 +\a f^2 + 2\a p \label{axi2}\\
(1-\a) f a + a a' - b^2 \cot \phi & = - p'  \label{axi3}\\
(1-\a) f b + a b' + a b \cot \phi  & = 0. \label{axi4}
\end{align}
\end{subequations}

System \eqref{e:axi} has two conserved quantities. First, when $\a \neq 2$ we can express $f$ in terms of $a$ from the \eqref{axi1}, plug into \eqref{axi4}, divide by $ab$, provided $ab \neq 0$, we obtain the law
\begin{equation}\label{e:abA}
|b|^{2-\a} |a|^{\a-1} \sin \phi = A.
\end{equation}	
Second can be obtained from \eqref{e:H}. That equation in the axisymmetric case takes form
\begin{equation}\label{aHf}
a H' = 2\a f H.
\end{equation}
Let us suppose that $aH \neq 0$ on some interval $\phi \in I$. Then the above implies $\frac{d}{d\phi} \ln |H| = 2 \a \frac{f}{a}$. From \eqref{axi1} we also obtain $ \frac{d}{d\phi} \ln |a \sin \phi| = (\a-2) \frac{f}{a}$. We thus recover a closed differential which implies
\begin{equation}\label{aH}
|H|^{2-\a}|a \sin \phi|^{2\a} = B.
\end{equation}
We now obtain several results with the use of the found conservation laws.
\begin{proposition}\label{p:noaxi} There are no $C^2$ axisymmetric solutions in the range $0<\a<2$.
\end{proposition}
\begin{proof} If $aH \neq 0$ on some interval $I$, then we immediately obtain from \eqref{aH} that $I = (0,\pi)$, and since $\sin \phi$ vanishes, $H$ becomes unbounded, which is a contradiction. Then $aH=0$ everywhere. Suppose $H \neq 0$ on some interval $I$. Then $a=0$, and from \eqref{aHf}, $f=0$. The entire system reduces to \eqref{e:b=0} with explicit solutions \eqref{e:rot}. These imply that $I = (0,\pi)$ since $H$ stays bounded away from zero. Hence $H$ blows up, which is a contradiction. We have proved that $H = 0$ on the entire sphere. By \prop{p:irrot} such solutions are irrotational and $\a$ is an integer, which excludes solutions in the given range.
\end{proof}

\begin{proposition}\label{l:axirad} The only axisymmetric solution available for $\a=2$ is the radial one given by \eqref{e:rad}. The only solutions available in the case $\a=0$ are the irrotational ones \eqref{e:irr}. 
\end{proposition}
\begin{proof} For the first part, from \eqref{axi1} we obtain
\[
a(\phi) = a(\phi_0) \frac{\sin \phi_0 }{\sin \phi}.
\]
So, unless $a = 0$ everywhere, we obtain a singular solution. If however $a=0$ everywhere, then \eqref{axi4} implies $fb = 0$. Suppose $f(\phi_0) \neq 0$, and hence by continuity $b = 0$ in a neighborhood of $\phi_0$. In that neighborhood $p'=0$ as implied by  \eqref{axi3}, so $p = p_0$,  a constant. Then $f = \const$ too. This implies that the condition $f \neq 0$ spreads to the entire sphere. Hence the solution is radial. The opposite case $f = 0$ is excluded by \prop{p:f=0}.

If $\a = 0$, then \eqref{aH} implies that unless $H=0$ throughout, $H$ must be constant. The description follows from \prop{p:irrot}.
\end{proof}

We now will give a complete description of solutions with constant spherical pressure $p$. It is not immediate that solutions are geodesic because the global pressure $P$ is not constant for $\a \neq 0$ unless $p=0$. However, the pressure does disappear from the momentum last two equations of \eqref{e:axi} which makes the classification possible. The general case remains open. 

\begin{proposition}\label{p:axip} Axisymmetric $C^1$-solutions with $p=\const$ are geodesic and are given by one of the solutions in the family \eqref{e:Vaxi} (in which case $\a \leq -2$ with swirl, and $\a\leq -1$ without), or by the radial solution \eqref{e:rad} in the case $\a =2$.
\end{proposition}
\begin{proof} Since the case $\a=2$ has been handled by \lem{l:axirad} we can assume $\a \neq 2$. Since we don't know a priori if $a$ or $b$ vanish somewhere, let us look into those cases separately.

Let us denote
\[
R(\phi,\phi_0) = \frac{\sin \phi }{\sin \phi_0}.
\]
Let us assume that at some $0<\phi_0<\pi$, $b(\phi_0) = 0$, no swirl. Then the orbit of $v$ through that point is a part of the corresponding meridian, and thus $b = 0 $ on that orbit. Solving \eqref{e:axi} we obtain  explicitly:
\begin{equation}\label{e:b=0,a=2}
a(\phi) = a(\phi_0) R^{1-\a}(\phi, \phi_0), \quad f = - a(\phi_0) \cot \phi R^{1-\a}(\phi, \phi_0), \quad \a p_0 = 0.
\end{equation}
This identifies the solution as a parallel shear flow \eqref{e:psf1} - \eqref{e:psf2} with constant $z$, which is a part of \eqref{e:Vaxi} family.

\begin{claim} If $a(\phi_0) = 0$ then $v(\phi_0) = 0$. 
\end{claim} Indeed, unless, $\phi_0 = \pi/2$, we have $b(\phi_0) = 0$ straight from the third of \eqref{e:axi}. If  $\phi_0 = \pi/2$, and if $b(\phi_0) \neq 0$, then the equator is the orbit. Pick a $\phi_n = \pi/2 + \frac{1}{n}$. For large $n$ by continuity $v(\phi_n) \neq 0$, so the orbit through $\phi_n$ is a non-trivial part of a geodesic. Clearly one end of this geodesic orbit must land at a latitude closer to the equator than the original $\phi_n$ (the geodesics cannot cross by uniqueness). At that point $\pi/2<\phi_n'<\pi/2 + \frac{1}{n}$, $b(\phi_n') = 0$. Taking the limit we have $b(\pi/2) = 0$, which is a contradiction.
Thus, in either case $a(\phi_0) = 0$ implies $v(\phi_0) = 0$. 

Now let us assume that $v_0=v(\phi_0) \neq 0$, and $b_0 \neq 0$. In this case  the entire system \eqref{e:axi} can be solved explicitly with help of \eqref{e:abA}. The computation is routine. We use \eqref{e:abA} to solve through \eqref{axi3} to obtain
\begin{equation}\label{e:ab}
\begin{split}
a(\phi) &= \frac{\sign(a(\phi_0)) }{|a(\phi_0)|^{1-\a} R(\phi,\phi_0)} \left[ |v(\phi_0)|^2 R^2(\phi,\phi_0)  - b^2(\phi_0) \right]^{\frac{2-\a}{2}} \\
b(\phi) &= \frac{b(\phi_0)}{|a(\phi_0)|^{1-\a} R(\phi,\phi_0)} \left[ |v(\phi_0)|^2 R^2(\phi,\phi_0)  - b^2(\phi_0) \right]^{\frac{1-\a}{2}},
\end{split}
\end{equation}
and plugging it into  \eqref{axi1} we find
\begin{equation}\label{e:f}
f(\phi) = -  \frac{\sign(a(\phi_0))|v(\phi_0)|^2 R(\phi,\phi_0)}{|a(\phi_0)|^{1-\a}} \cot \phi  \left[ |v(\phi_0)|^2 R^2(\phi,\phi_0)  - b^2(\phi_0) \right]^{-\frac{\a}{2}}.
\end{equation}
From \eqref{axi2} we finally find
\begin{equation}\label{e:ap}
\a p_0 = 0.
\end{equation}
The solution is valid as long as $|v(\phi_0)|^2 R^2(\phi,\phi_0)  - b^2(\phi_0)  >0$. This region in terms of $\phi$ is symmetric with respect to $\phi = \frac{\pi}{2}$. First, this means that there is only one band of geodesics in which $v \neq 0$. Second, resetting $\phi_0$ to $\frac{\pi}{2}$, and rescaling $v(\pi/2) = a_0 \ef + b_0 \eth$ to magnitude $1$, and rewriting \eqref{e:ab} - \eqref{e:f} in Cartesian coordinates we arrive precisely at \eqref{e:Vaxi}. Inside the cone the solution must vanish. This describes the solution completely.
\end{proof}
\begin{remark}
Finally, we remark that the with the help of first laws \eqref{e:abA} and \eqref{aH} the system \eqref{e:axi} reduces to a system of two ODEs, for example, on $(f,a)$. One can rewrite it as a Hamiltonian non-autonomous system. It could be possible to exclude solutions that are not already described in this section. For instance, solution without swirl satisfy 
\begin{equation}\label{}
\begin{split}
x' &= (\a-2) f \\
f' & = \a B |x|^{\frac{4}{\a-2}}x + (1-\a) \frac{x}{1-t^2},
\end{split}
\end{equation}
where $x = a \sin \phi$, and $t = - \cos \phi$, $-1 < t<1$. The Hamiltonian is given by $\mathcal{H}(t,x,f) = (a-2) f^2 +(1-\a) \frac{x^2}{1-t^2} + (2-\a) B |x|^{\frac{2\a}{\a-2}}$. It is a Lyapunov function for the system on intervals $(-1,0]$ and $[0,1)$. 
Numerical computations show that unless $\a$ is an integral and solution is irrotational corresponding to the central harmonic $f = Y^0_{1-\a}$, generically $x \neq 0$ at $t = \pm 1$, which implies that $a \ra \infty$, hence excluded as non-smooth. We will perform more close analysis of this case in the near future.
\end{remark}

\subsection{Relation to Landau solutions}\label{s:landau} Even though for $\a=1$ there are no smooth solutions, for the Navier-Stokes equation the scaling of $\a=1$ is the only one possible. Axisymmetric homogeneous solutions for Navier-Stokes were found by Landau in his little known paper \cite{landau}, see also Batchelor's text \cite{batchelor} with physical insight into Landau solutons. They have been revisited recently in the work of Sverak \cite{sverak}, who showned that \emph{any} smooth homogeneous solutions for the Navier-Stokes equation are Landau. The proof uses maximum principle to find that $v$ is irrotational and the potential function $\f$, $v = \n \f$, satisfies a constant curvature equation for a conformally equivalent metric. The corresponding (anti)conformal transformation of the sphere given by a conjugate to the simple scalar multiplication via the stereographic projection yields the explicit solution of Landau.  One might consider the question of vanishing viscosity limit in which a possibility exists of obtaining singular solutions to the Euler system from smooth solutions to the Navier-Stokes. Unfortunately this is not the case. Let us discuss it in more detail.

We consider axisymmetric solutions without swirl for $\a=1$. So, we let all ingredients depend only on $\phi$, and $v^\th = 0$. Consider the Stokes stream-function $\psi(\phi)$:
\begin{equation}\label{e:axiStokes}
f = \frac{1}{\sin \phi} \psi', \quad a = -\frac{1}{\sin \phi} \psi.
\end{equation}
Then $2p + (v^\phi)^2 = \const$, and the system \eqref{e} integrates into
\begin{equation}\label{e:x}
\psi^2 = Ax^2+ Bx+C,
\end{equation}
where $x = \cos \phi$, and $A,B,C\in\R$. To ensure positivity of the right hand side of \eqref{e:x}, we have 
\[
\left\{ \begin{aligned}
B^2 &\leq 4AC \\
C &\geq 0
\end{aligned}\right.
\text{ or }
\left\{ \begin{aligned}
|B| &\leq A+C \\
|B| &\geq 2A
\end{aligned}\right. .
\]
This gives a family of axisymmetric solutions, expectedly singular. More directly, viewing $\psi$ as a function of $x$, in order for \eqref{e:x} to give smooth functions we need $\psi(\pm 1) = \psi'(\pm 1) = 0$, which yields $A=B=C=0$. Let us recall that the Landau solutions satisfy (see Bachelor \cite{batchelor} eq. (4.6.8)):
\begin{equation}\label{}
\psi^2 -2 \nu (1-x^2) \psi' - 4\nu x \psi = Ax^2+ Bx+C,
\end{equation}
where $\nu >0$ is viscosity, and $\psi'$ is with respect to $x$. As argued in \cite{batchelor}, unless $A=B=C=0$ the solutions are singular as well. So, the only way to restore solutions to Euler via vanishing viscosity limit is through a sequence of singular solutions. Otherwise, smooth Landau solutions converge to trivial $0$ as $\nu \ra 0$.

\section{Relation to Onsager's conjecture}

We cannot rule out smooth solutions in many scalings, among which the case $\a = \frac{2}{3}$ stands out. In this case the field $V$  lends itself into the so-called Onsager-critical homogenous Besov space $\dot{B}^{1/3}_{3,\infty}$. This field therefore provides a candidate for energy flux anomaly, whose existence is asserted in the classical Onsager's conjecture. The globally homogeneous field $V$, however, shows $1/3$ critical smoothness both at the small scales, namely at the origin, and at the large scales, namely at infinity. Moreover it belongs to no $L^p$-space in $\R^3$. We will therefore modify the field $V$ in order to only create a solution with small scale singularity at the origin, locally $C^\infty$ away from the origin, and with a compact support. This field, denoted $\bar{V}$, along with the associated pressure $\bar{P}$ will satisfy a forced Euler system with force $F \in C^\infty_{\loc}(\R^3)$ and $F \sim 1/|x|^3$ at infinity. The new field $\bar{V}$ has globally finite energy, we investigate a possibility for the energy flux anomaly. The anomaly occurs when for such a solution we have a non-zero work of force (while being stationary),
\begin{equation}\label{e:flux}
\Pi = \int_{\R^3} F \cdot \bar{V} dx \neq 0.
\end{equation}
In order to properly truncate the field $V$ while preserving the divergence-free condition we will make use of a stream-field, analogous to stream-function in 2D.

\subsection{Stream-field} Despite that $V$ is defined on a non-simply connected domain we can still construct, at least for any $\a \neq 2$, a so-called stream-field $\Psi$ satisfying
\begin{equation}\label{e:curl}
V = \curl \Psi, \quad \Psi = \frac{1}{|x|^{\a-1}}( \psi + h \nor), \quad \diver \Psi = 0,
\end{equation}
where $\psi$ is the tangential and $h$ is the vertical components. The system  \eqref{e:curl} is equivalent to
\begin{subequations}\label{e:curlsys}
\begin{align}
v &= (2-\a) \psi^\perp - \n^\perp h \label{e:curlsys1}\\
f &= \star d \psi \label{e:curlsys2}\\
(3-\a) h + \diver \psi &= 0. \label{e:curlsys3}
\end{align}
\end{subequations}
Let us focus on the first two equations first. Since $(2-\a) f = - \diver v$, and $\a \neq 2$ we see that $\int f d\s = 0$, which means that, as on any compact orientable manifold, the form $f d\s$ is exact. So, there is $\psi$ so that $f d\s= d \psi$. This satisfies \eqref{e:curlsys2}. Using \eqref{e1} we have 
\[
\d (v -(2-\a) \psi^\perp) = -(2-\a) f - (2-\a) \star d \star \star \psi = -(2-\a) f + (2-\a) \star d \psi = 0.
\]
Hence, $v -(2-\a) \psi^\perp$ is co-exact as a form. This implies the existence of $h$ to satisfy \eqref{e:curlsys1}. Now that the first two equations in \eqref{e:curlsys} being satisfied, let us notice that $\psi$ can be changed by an exact form, i.e. $\psi + d\f$ will do as well, for any $\f$, and $h$ can be changed by a constant. Adjusting $h$ by a constant to satisfy 
\[
\int ((3-\a)h + \diver \psi) d\s = 0,
\]
we can guarantee that the Poisson equation
\[
\D \f = -(3-\a)h - \diver \psi
\]
has a solution. With the new $\psi = \psi_{old} + d\f$ this implies 
\eqref{e:curlsys3}, i.e. $ \Psi$ is divergence-free on $\R^3 \bs \{0\}$.

\subsection{Tapering the field} Let $V, P$ be given by \eqref{e:VP}, $\a\neq 2$, and let $\Psi$ be a stream-field of $V$. Let $\f(r)$ be given by $\frac{1}{r^{\a -1}}$ for $r<1$, $\f = 0$ for $r >2$, and $\f$ be radial and smooth in the ring $1\leq r \leq 2$. Let 
$\bar{\Psi} = \f ( \psi + h \er)$ and $\bar{V} = \curl \bar{\Psi}$. Finally, let $\tilde{P} = \f P$. Clearly, the pair $(\bar{V}, \tilde{P})$ is supported within $r\leq 2$, and coincides with $(V,P)$ in the unit ball. This implies in particular that $(\bar{V}, \tilde{P})$ satisfies the same Euler system in the unit ball. We now find a global pressure $\bar{P}$ which complements the pair $(\bar{V},\bar{P})$ to a solution on the whole space but with additional smooth divergence-free force $F$:
\begin{equation}\label{e:barE}
\begin{split}
\bar{V} \cdot \n \bar{V} + \n \bar{P} & = F \\
\diver \bar{V} & = 0.
\end{split}
\end{equation}
We will look for $\bar{P}$ in the form $\bar{P} = \tilde{P} + P_0$, where $P_0$ is a corrector pressure to be found. Taking the divergence of \eqref{e:barE} we read off  the following Poisson equation for $P_0$:
\begin{equation}\label{}
\D P_0 = Q = \left\{
\begin{split}
&0, & r<1\\
&- \D \tilde{P} - \diver \diver  \bar{V} \otimes \bar{V}, & 1 \leq r \leq 2\\
&0, & r >2.
\end{split}\right.
\end{equation}
Thus, a solution is given by the classical convolution with the Newton potential, while the gradient satisfies
\begin{equation}\label{}
\n P_0 (x)= c \int_{\R^3} \frac{x-y}{|x-y|^3} Q(y) dy.
\end{equation}
Note that $P_0$ is locally a $C^\infty$ function, as $Q$ is. Moreover, $Q$ is mean-zero,
\begin{equation}\label{}
\int Q(y) dy = \int_{1\leq |y| \leq 2} \diver( - \n \tilde{P} - \bar{V} \cdot \n \bar{V}) dy = \int_{\S^2} (\n \tilde{P} + \bar{V} \cdot \n \bar{V}) \cdot \nu d\s = 0,
\end{equation}
the latter being trivial in view of $(\tilde{P}, \bar{V})$ satisfying the Euler equation pointwise on the sphere. Therefore, for large $x$ we have
\[
\n P_0(x)= c \int_{\R^3}\left( \frac{x-y}{|x-y|^3} - \frac{x}{|x|^3} \right) Q(y) dy \sim \frac{1}{|x|^3}.
\]
Similarly, $\n^k P_0(x) \sim \frac{1}{|x|^{3+k}}$ for all $k \in \N$. We thus see that the pair $(\bar{V},\bar{P})$ satisfies \eqref{e:barE} with $F$ being
\begin{equation}\label{e:F}
F = \left\{ 
\begin{split}
& \n P_0,& r<1 \\
&\bar{V} \cdot \n \bar{V} + \n \bar{P},& 1\leq r \leq 2 \\
&\n P_0,& r>2.
\end{split}\right.
\end{equation}
So, $F \in C^\infty_\loc (\R^3)$ and in addition 
\begin{equation}\label{}
\n^k F(x) \sim \frac{1}{|x|^{3+k}} ,\text{ for all } k=0,1,\ldots,  \text{ as } x \ra \infty.
\end{equation}
This lands the force into the natural Sobolev spaces $W^{k,p}$ for all  $p>1$.

\subsection{Absence of flux anomaly} Let $\a = \frac23$. We have a solution to the Euler system \eqref{e:barE} with a smooth decaying force and point singularity  at the origin and $\bar{V} \in B^{1/3}_{3,\infty}(\R^3)$ with compact support. Let us find a formula for the flux \eqref{e:flux}. From the formula for the force \eqref{e:F} via integration by parts we obtain,
\begin{multline}
\int F \cdot \bar{V} dx = \int_{|x| <1} \n P_0 \cdot V dx + \int_{1\leq |x| \leq 2} (\frac12 \bar{V} \cdot \n |\bar{V}|^2 + \n (P_0 + \tilde{P}) \cdot \bar{V} )dx \\
= \int_{|x|=1} P_0 V \cdot \nu - \int_{|x|=1}(P_0 + P) V \cdot \nu - \int_{|x|=1} \frac12 |V|^2 V \cdot \nu = - \frac12 \int_{\S^2} f H d\s.
\end{multline}
Thus,
\begin{equation}\label{}
\Pi = - \frac12 \int_{\S^2} f H d\s.
\end{equation}
\begin{lemma}\label{l:fH}
For any $\a \in \R$ and smooth solution \eqref{e:VP} we have
\begin{equation}\label{e:fH}
\int_{\S^2} f H^n d\s = 0,
\end{equation}
for all $n \in \N$, and even for $n = 0$ if $\a \neq 2$.
\end{lemma}
Incidentally, the case of interest $\a = \frac23$, $n = 1$ appears to be critical in the following proof of the lemma. Clearly, if $\a=2$, the radial solution is a counterexample for \eqref{e:fH}, $n=0$.
\begin{proof} Multiplying \eqref{H} with $H^{n-1}$, $n \in \N$ and integrating over the sphere we obtain $\int f H^n d\s= 0$, for all $n \in \N$ except a possible $n_0$ for which $\a = \frac{2}{1+2n_0}$. To prove the identity for $n=n_0$ we argue as follows. We have $\int f H^{n_0} H^k d\s = 0$,
for all $k =1,2,\ldots$. Consequently, $\int f H^{n_0} G(H) d\s = 0$, for all real analytic functions $G$ with $G(0) = 0$. Letting $G(x) = 1- e^{-x^2/\e}$ and letting $\e \ra 0$ we obtain
$\int_{H \neq 0} f H^{n_0} d\s = 0$. However, on the set $\{H = 0\}$ the integral vanishes trivially. 

When $\a \neq 2$ we also have $\int f d\s = 0$ directly from \eqref{e1}.
The lemma is proved.
\end{proof}

\begin{remark}
Multiplying \eqref{e:H3} with $u$ we obtain
\begin{equation}\label{Hu}
u\n H = 2\a  \o H.
\end{equation}
Similarly to the argument above, we also have
\begin{equation}\label{e:gH}
\int \o H^n d\s= 0,
\end{equation}
for all $n\in \N$, and since $\o d \s = d v$ we have \eqref{e:gH} for $n=0$ by Stokes' Theorem.
\end{remark}

\section{Appendix: glossary of terms}  All facts from differential geometry used  in the text can be found, for instance, in \cite{rosen}. System \eqref{e} can be easily derived from \eqref{e:ES} by applying the following formulas (see also \cite{sverak}). If $u,v \in T\S^2$ and $f \in C^1(\S^2)$ are $0$-homogeneous on $\R^3\bs\{0\}$, then 
\[
\begin{split}
\n_{\R^3} (f / |x|^\a) & = \frac{1}{|x|^{\a+1}} (\n_{\S^2} f - \a f )\, \nor\\
u \cdot \n_{\R^3} v& = \frac{1}{r} (u \n_{\S^2} v - (u \cdot v) \nor ) \\
v \cdot \n_{\R^3} (f \nor) &= \frac{1}{r} ( vf+ (v\n_{\S^2} f )\, \nor) .
\end{split}
\]
Recall the Riemannian metric tensor $g = \sin^2\phi\, d\th^2 + d\phi^2$.  Let us write $v =  v^\phi \p_\phi + v^\th \p_\th $ in local spherical coordinates. The transformation formulas into the unit coordinate frame $v = a \ef + b \eth$ are
\begin{equation}
a =  v^{\phi}, \quad b = \sin \phi\, v^\th.
\end{equation}
The dual form to $v$ is given by $v^\flat = (\sin\phi)^2 v^\th d\th + v^\phi d\phi = b \sin \phi\, d\th + a \, d\phi$. The 2D ``vorticity" discussed in the text is given by the scalar function $\o = \star d v^\flat$, where $\star$ is the Hodge-star operation. Thus, $dv^\flat = \o d\Vol$, where $d\Vol = \sin\phi\, d\phi \wedge d\th$.  So, $\o = b_\phi + b \cot \phi - a_\th (\sin\phi)^{-1}$. We adopt the 1D adjoint to $d$, $\d = \star d \star$, so that $\d v^\flat = \diver v$. Finally, for a scalar function $f$ on $\S^2$ we use negative definite Laplacian $\D f = \d d f$.



\end{document}